\documentclass[final,nomarks]{dmtcs-episciences}

\usepackage[T1]{fontenc}
\usepackage{amsmath, amsfonts, amssymb, amsthm, dsfont, hyperref, mathrsfs, shuffle, url, wasysym} 

\usepackage{tikz}

\graphicspath{{figures/}} 
\makeatletter\def\input@path{{figures/}}\makeatother

\usepackage{caption}
\usepackage[noabbrev,capitalise]{cleveref}

\newtheorem{theorem}{Theorem}[section]

\newtheorem{proposition}[theorem]{Proposition}
\newtheorem{lemma}[theorem]{Lemma}

\theoremstyle{definition}
\newtheorem{definition}[theorem]{Definition}

\newtheorem{remark}[theorem]{Remark}

\crefname{notation}{Notation}{Notations}
\crefname{problem}{Problem}{Problems}

\newcommand{\T}{\mathcal{T}} 
\newcommand{\A}{\mathcal{A}} 

\newcommand{\tdot}[3]{\draw [fill=black,color=#3] (#1,#2) circle [radius=0.25];}

\newcommand{\x}{\mathbf{x}}

\DeclareMathOperator{\inv}{inv} 
\DeclareMathOperator{\sgn}{sgn} 

\usepackage{color}
\definecolor{orange}{rgb}{0.898, 0.621, 0.0}
\definecolor{skyblue}{rgb}{0.336, 0.703, 0.910}
\definecolor{bluishgreen}{rgb}{0, 0.617, 0.449}
\definecolor{yellow}{rgb}{0.937, 0.890, 0.258}
\definecolor{newblue}{rgb}{0, 0.245, 0.395}
\definecolor{red}{rgb}{0.832, 0.367, 0}
\definecolor{purple}{rgb}{0.797, 0.473, 0.652}

\usepackage{todonotes}

\usepackage[backend=bibtex, firstinits=true, maxbibnames=99]{biblatex}
\DeclareFieldFormat[article,inbook,incollection,inproceedings,patent,thesis,unpublished]{title}{\textit{#1}\isdot}
\renewbibmacro{in:}{}

\title{About the determinant of complete non-ambiguous trees}

\author{Jean-Christophe Aval}
\affiliation{LaBRI, CNRS\\ Universit\'e de Bordeaux\\ France}

\begin{document}
\publicationdata{vol. 26:3}{2024}{18}{10.46298/dmtcs.12850}{2024-01-10; 2024-01-10; 2024-08-08}{2024-10-17}
\maketitle
\begin{abstract}
$ $\\
Complete non-ambiguous trees (CNATs) are combinatorial objects which appear in various contexts.
Recently, Chen and Ohlig studied the notion of permutations associated to these objects, 
and proposed a series of nice conjectures.
Most of them were proved by Selig and Zhu, through a connection with the abelian sandpile model.
But one conjecture remained open, about the distribution of a natural statistic named determinant.
We prove this conjecture, in a bijective way.
\end{abstract}
\section{Introduction}
Non-ambiguous trees (NATs) were first defined in \cite{NAT} and may be seen as a proper way to draw a binary tree on a square grid (see Definition~\ref{def:nat}). They were put to light as a special case of tree-like tableaux, which have been found to have applications in the PASEP model of statistical mechanics \cite{CW,TLT}. 
The initial study of NATs revealed nice properties, mostly in an enumerative context \cite{NAT,NAT2}. This includes enumeration formulas with respect to fixed constraints, and new bijective proofs of combinatorial identities. When the undelying binary tree is {\em complete}, we are led to complete non-ambiguous trees (CNATs).
These objects were first considered in \cite{NAT}, where it was proved that their enumerating sequence is related to the formal power series of the logarithm of the Bessel function of order $0$. An extension to higher dimension was proposed in \cite{Patxi2}.

Recent papers have revealed new facets of these objects. 
In \cite{DGGS19}, striking mathematical cross-connections were obtained, such as a bijection between CNATs
and fully-tiered trees of weight $0$.
In \cite{DSSS18}, CNATs were linked to the abelian sandpile model. In the same article, it was noticed that if we restrict a CNAT to its leaf dots, we obtain a permutation. This link was investigated in \cite{CO23}, where nice properties were derived, and several conjectures proposed. By using the connection with the abelian sandpile model,
a large number of conjectures were proved very recently in \cite{SZ23},
but one conjecture remained open.
It asserted that when considering the set of CNATs of a fixed odd size, 
the number of them having an underlying permutation with even and odd determinant (signature) are equal.
We give a bijective proof of this (Theorem~\ref{theo:main}), and include the case of even size,
which was suggested in~\cite{CO23}.

\section{Definitions and statement of the result}

We first recall the definition of (complete) non-ambiguous trees, as in~\cite{NAT}.

\begin{definition}\label{def:nat}
A \emph{non-ambiguous tree} (NAT) $T$ is a filling of an $m \times n $ rectangular grid, where each cell is either dotted or not, satisfying the following conditions:
\begin{description}
\item[Existence of a root] The top-left cell is dotted; we call it the {\em root} of $T$.
\item[Non ambiguity] Aside from the root, every dotted cell of $T$ has either a dotted cell above it in the same column, or a dotted cell to its left in the same row, but not both.
\item[Minimality] Every row and every column of $T$ contains at least one dotted cell.
\end{description}
\end{definition}

\begin{remark}
The use of the word \emph{tree} to describe these objects comes from the following observation. 
Given a NAT $T$, we connect every dot $d$ different from the root to its \emph{parent} dot $p(d)$, which is the dot immediately above it in the same column, or to its left in the same row (because of the condition of non ambiguity, exactly one of these must exist). 
\end{remark}

\begin{definition}\label{def:cnat}
A \emph{complete non-ambiguous tree} (CNAT) is a NAT
whose underlying tree is complete, {\it i.e.}\ in which every dot has either both a dot below it in the same column and a dot to its right in the same row (in which case the dot is said to be an {\em internal dot}), or neither of these (in which case the dot is said to be a {\em leaf}). 

The \emph{size} of a CNAT is its number of leaf dots, or equivalently one more than its number of internal dots. 

We denote by $\T_n$ the set of CNATs of size $n$ and $T_n=|\T_n|$.
\end{definition}
Figure~\ref{fig:cnat} gives an example of this notion.

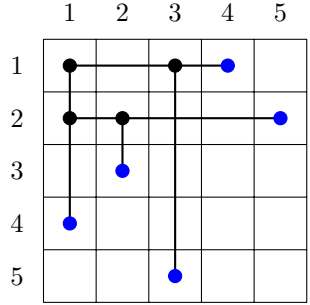
\begin{figure}[ht]
\begin{center}
\begin{tikzpicture}[scale=0.35]

\draw [step=2] (2,2) grid (12,-8);
\foreach \x in {1,...,5}
  \node at (1+2*\x, 3) {$\x$};
\foreach \y in {1,...,5}
  \node at (1, 3-2*\y) {$\y$};
\draw [thick] (5,-3)--(5,-1)--(11,-1);
\draw [thick] (3,-5)--(3,1)--(9,1);
\draw [thick] (7,1)--(7,-7);
\draw [thick] (3,-1)--(5,-1);
\tdot{3}{1}{black}
\tdot{7}{1}{black}
\tdot{3}{-1}{black}
\tdot{5}{-1}{black}
\tdot{3}{-5}{blue}
\tdot{5}{-3}{blue}
\tdot{7}{-7}{blue}
\tdot{9}{1}{blue}
\tdot{11}{-1}{blue}

\end{tikzpicture}
\end{center}
\caption{A CNAT of size $5$. Leaf dots are represented in blue, and internal dots in black. \label{fig:cnat}}
\end{figure}

As in this figure, it will be convenient to label by integers the rows and columns respectively from top to bottom 
and from left to right, in such a way that the root appears in the cell $(1,1)$. 
Moreover, given a dot $d$ in a CNAT, we denote by $c(d)$ and $r(d)$ the label of its column and row.
For a given internal dot, its child in the same row is called its {\em right child} 
and its child in the same column is called its {\em left child}.

\begin{remark}\label{rema:alone}
We may observe that any right leaf $l$ in a CNAT $T$ is the only dot in its column:
there is no dot above $l$ because this would contradict the minimality condition of Definition~\ref{def:cnat},
and there is no dot below $l$ because $l$ is a leaf.
In the same way, any left leaf $l$ in $T$ is the only dot in its row.
\end{remark}

We may see a CNAT $T$ as a matrix $M(T)$ where dotted cells are $1$'s and undotted cells are $0$'s
For example, the CNAT of Figure~\ref{fig:cnat} is encoded by the following matrix: 
$$
\left(
\begin{array}{ccccc}
1&0&1&1&0\cr
1&1&0&0&1\cr
0&1&0&0&0\cr
1&0&0&0&0\cr
0&0&1&0&0\cr
\end{array}
\right).
$$

The number $T_n$ of CNATs of size $n$ appears as the series {\tt A002190} in \cite{oeis}.
As proved in \cite{NAT}, these numbers give a combinatorial interpretation for the development of 
the Bessel function $J_0$.

Let us now introduce the notion of permutation associated to a CNAT.
\begin{definition}
Let $T$ be a CNAT of size $n$.
It is clear that in any column of $T$ the bottom-most dot of is a leaf, as well as the right-most dot of any row.
Thus every row and every column must have exactly one leaf dot.
As such, the set of leaf dots of a CNAT $T$ of size $n$ forms the graphical representation of an $n$-permutation.
We define $\pi(T)$ as the permutation whose $j$-th entry is the row index of the bottommost 
dot that appears in column $j$, and call it the permutation associated to the CNAT $T$.
\end{definition}

For example, the CNAT of Figure~\ref{fig:cnat} has associated permutation $\pi(T) = 43512$.

\begin{definition}
We recall that an {\em inversion} in a permutation $\sigma$ is a pair of entries
$\sigma_i > \sigma_j$ with $i<j$.
The {\em sign} of $\sigma$ is defined by:  $\sgn(\sigma)={(-1)}^{\inv(\sigma)}$, where $\inv(\sigma)$ is the number of inversions of $\sigma$.
\end{definition}

A careful study of permutations associated to CNATs was initiated in \cite{CO23},
where the following proposition was proved.
\begin{proposition}
Let $T$ be a CNAT.
We have:
$$\det M(T) = \sgn \pi(T).$$
\end{proposition}

For $\epsilon=\pm 1$, let us denote by $T(n;\epsilon)$ the number of CNATs of size $n$ with determinant equal to $\epsilon$.
We are now in a position to state the main result of this article.
\begin{theorem}\label{theo:main}
If $n>1$ is odd: 
\begin{equation}
T(n;+1) = T(n;-1) = {T_n \over 2}.
\end{equation}
If $n$ is even (let us set $n=2p$): 
\begin{equation}
T(2p;+1) = {T_{2p}+{(-1)}^{p}T_p \over 2} \text{ and } T(2p;-1) = {T_{2p}-{(-1)}^{p}T_p \over 2}.
\end{equation}
\end{theorem}

The odd case corresponds to Conjecture~2.6 in \cite{CO23}, the even case to Remark~2.7 in the same paper.

\section{A bijective proof of Theorem~\ref{theo:main}}

This section is devoted to proving our main result.
This proof is bijective.
More precisely, we shall:
\begin{enumerate}
\item introduce a subset $\A_{2p}\subset\T_{2p}$ of CNATs of even size, with $|\A_{2p}|=T_p$, and such that for any $T\in\A_{2p}$ we have $\sgn \pi(T) = {(-1)}^p$;
\item construct an {\em involution} $\Phi$ on the set of CNATs such that if 
$T$ is not in any of the sets $\A_{2p}$ we have: 
$$\sgn \pi(\Phi(T)) = -\sgn \pi(T).$$
\end{enumerate}

We first introduce a useful notion on the leaves of CNATs.
\begin{definition}
A leaf in a CNAT is said to be {\em short} if its parent is in a cell adjacent to it.
Otherwise the leaf is said to be {\em long}.
Moreover, we denote by $\A_n$ the set of CNATs of size $n$ with only short leaves.
\end{definition}
Figures~\ref{fig:short} illustrates this notion.

\begin{figure}[ht]
\begin{center}
\begin{tikzpicture}[scale=0.35]

\draw [step=2] (2,2) grid (14,-10);
\draw [thick] (5,-3)--(5,-1)--(13,-1);
\draw [thick] (3,-5)--(3,1)--(9,1);
\draw [thick] (7,1)--(7,-9);
\draw [thick] (3,-1)--(5,-1);
\draw [thick] (11,-1)--(11,-7);
\tdot{3}{1}{black}
\tdot{7}{1}{black}
\tdot{3}{-1}{black}
\tdot{5}{-1}{black}
\tdot{11}{-1}{black}
\tdot{3}{-5}{red}
\tdot{5}{-3}{blue}
\tdot{7}{-9}{red}
\tdot{9}{1}{blue}
\tdot{11}{-7}{red}
\tdot{13}{-1}{blue}

\begin{scope}[shift={(22,0)}]

\draw [step=2] (2,2) grid (14,-10);
\draw [thick] (3,-9)--(3,1)--(9,1);
\draw [thick] (3,-3)--(13,-3);
\draw [thick] (3,-7)--(5,-7);
\draw [thick] (7,1)--(7,-1);
\draw [thick] (11,-3)--(11,-5);
\tdot{3}{1}{black}
\tdot{7}{1}{black}
\tdot{3}{-3}{black}
\tdot{11}{-3}{black}
\tdot{3}{-7}{black}
\tdot{3}{-9}{blue}
\tdot{5}{-7}{blue}
\tdot{7}{-1}{blue}
\tdot{9}{1}{blue}
\tdot{11}{-5}{blue}
\tdot{13}{-3}{blue}

\end{scope}
\end{tikzpicture}
\end{center}
\caption{(Left) A CNAT with short leaves in blue and long leaves in red.\\
(Right) An element of $\A_6$. 
\label{fig:short}}
\end{figure}
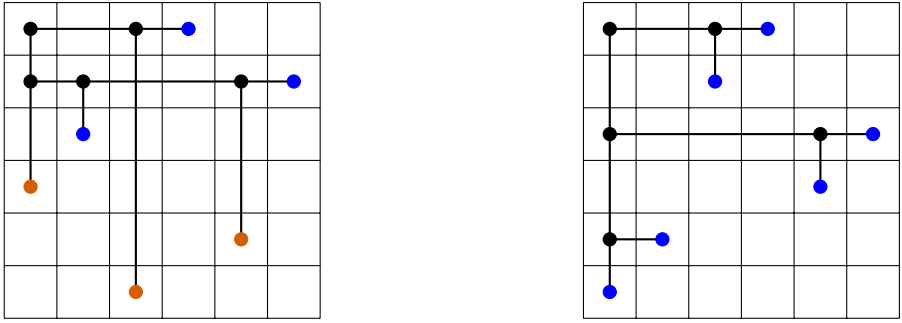

The elements of $\A_p$ are designed to be the fixed points of our involution $\Phi$.
We treat their case with two propositions, for which the technical key point is the following lemma.

\begin{lemma}\label{lemm:Ap-crux}
Let $T$ be a CNAT.
Suppose that $T$ has an internal dot with a leaf and an internal dot as children.
Then $T$ contains at least one long leaf.
\end{lemma}
\begin{proof}
Let us consider the set $C$ of internal dots with a leaf and an internal dot as children.
Among $C$, we consider an element $c$ which has no descendant in $C$,
{\it i.e. } there is no element in $C$ lower than $c$ in the tree.
By symmetry, we suppose that $c$ has an internal dot as right child, and a leaf as left child.
We refer to Figure~\ref{fig:Ap-crux} which shows only the part of $T$ of interest for the proof.

\begin{figure}[ht]
\begin{center}
\begin{tikzpicture}[scale=0.35]

\draw [step=2] (2,2) grid (16,-2);
\draw [dashed,thick] (3,1)--(11,1);
\draw [thick] (11,-1)--(11,1)--(13,1);
\draw [thick] (5,-1)--(5,1);
\tdot{5}{1}{black}
\tdot{11}{1}{black}
\tdot{13}{1}{blue}
\tdot{11}{-1}{blue}
\tdot{5}{-1}{blue}
\node at (5.5,1.5) {$c$};
\node at (11.6,1.6) {$c'$};
\end{tikzpicture}
\end{center}
\caption{Proof of Lemma~\ref{lemm:Ap-crux}.} 
\label{fig:Ap-crux}
\end{figure}
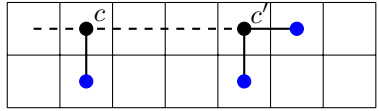

Consider the right-most internal dot $c'$ in the same row as $c$, which implies that 
its right child is a leaf.
By hypothesis, $c'$ is not in $C$, thus its left child is a leaf.
Then it is impossible for $c$ and $c'$ to have both a short left leaf, 
because these two leaves would lie in the same row.
\end{proof}

\begin{proposition}\label{prop:Ap-card}
When the size $n=2p+1$ is odd, the set $\A_{2p+1}$ is empty.
When the size $n=2p$ is even, the set $\A_{2p}$ is in bijection with $\T_p$.
\end{proposition}
\begin{proof}
Let us first consider the case of odd size, and let $T$ be a CNAT of size $n=2p+1$. 
Since $T$ has an odd number of leaves, there must be an internal dot in $T$ which has as children
an internal dot and a leaf.
Because of Lemma~\ref{lemm:Ap-crux}, $T$ contains at least one long leaf.
Thus $\A_{2p+1}$ is empty.

Now, let us suppose that the size is even.
We consider an element $T$ of $\A_{2p}$.
Lemma~\ref{lemm:Ap-crux} implies that any internal dot of $T$ has as children: 
either two internal dots, or two (short) leaves. 
Thus $T$ has exactly $p$ internal dots with two short leaves.
Now we shall erase all the leaves to get an element $T'\in \T_p$.
Thanks to Remark~\ref{rema:alone}, we may erase every right leaf in $T$ together with its column,
and every left leaf in $T$ together with its row.
Let us call this operation $R$.
By doing this, we get an element $T'=R(T)\in \T_p$, see Figure~\ref{fig:reduce}.

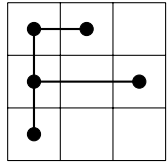
\begin{figure}[ht]
\begin{center}
\begin{tikzpicture}[scale=0.35]

\draw [step=2] (2,2) grid (8,-4);
\draw [thick] (3,-1)--(7,-1);
\draw [thick] (3,-3)--(3,1)--(5,1);
\tdot{3}{1}{black}
\tdot{3}{-3}{black}
\tdot{3}{-1}{black}
\tdot{5}{1}{black}
\tdot{7}{-1}{black}

\end{tikzpicture}
\end{center}
\caption{The operation $R$ applied to the element of $\A_6$  in Figure~\ref{fig:short}~(Right).
\label{fig:reduce}}
\end{figure}

We claim that $R$ is a bijection from $\A_{2p}$ to $\T_p$.
Let us describe the reverse bijection.
Consider an element $T'\in\T_p$. 
For every leaf $l'$ in $T'$,
we first add an empty column just to the right of $l'$
and an empty row just below $l'$,
and then two leaves as the children of $l'$.
It is clear that by doing this, we get an element $T\in\A_{2p}$ such that $R(T)=T'$.
\end{proof}

\begin{proposition}\label{prop:Ap-sgn}
Let $T$ be an element of $\A_{2p}$, its determinant is given by:
$$\sgn \pi(T) = {(-1)}^p.$$
\end{proposition}
\begin{proof}
We shall use the notations used in the proof of Proposition~\ref{prop:Ap-card},
and consider $T\in\A_{2p}$ and $T'\in\T_p$ with $T'=R(T)$.
We shall also set: $\sigma=\pi(T)$ and $\sigma'=\pi(T')$.
For the example of Figures~\ref{fig:short} and~\ref{fig:reduce}, we have:
$\sigma=436521$ and $\sigma'=231$.
We observe that $\sigma$ and $\sigma'$ are closely related.
If $\sigma'=\sigma'_1\sigma'_2\dots\sigma'_p$,
then 
$$\sigma=(2\sigma'_1)(2\sigma'_1-1)(2\sigma'_2)(2\sigma'_2-1)\dots(2\sigma'_p)(2\sigma'_p-1).$$
Thus any inversion $j>i$ in $\sigma'$ gives rise to four inversions in $\sigma$: $(2j-1)>(2i-1)$, $(2j)>(2i-1)$, $(2j-1)>(2i)$, $(2j)>(2i)$.
To which we have to add $p$ inversions: 
$(2\sigma'_1)>(2\sigma'_1-1)$, $(2\sigma'_2)>(2\sigma'_2-1)$, $\dots$ $(2\sigma'_p)>(2\sigma'_p-1)$.
Thus we are led to the following relation:
$$\inv \sigma = 4 \inv \sigma' + p$$
which implies that $\sgn \sigma = {(-1)}^{p}$.
\end{proof}

We now come to the definition of a function $\Phi$ on $\T_n$, which is the key construction of this work.
We first introduce the following notion.
\begin{definition}\label{def:inter}
Let $T$ be a CNAT. Let $l_1$ and $l_2$ be two leaves in $T$ with respective parents $p_1$ and $p_2$.
If $l_1$ and $l_2$ are both left leaves, they are said to be {\em interacting} if 
$$r(p_1)<r(l_2)<r(l_1) \text{ or } r(p_2)<r(l_1)<r(l_2).$$
The definition is similar for right leaves.
\end{definition}
This notion is illustrated by~Figure \ref{fig:inter}.
\begin{figure}[ht]
\begin{center}
\begin{tikzpicture}[scale=0.35]

\draw [step=2] (2,2) grid (14,-4);
\draw [dashed,thick] (11,1)--(11,-1);
\draw [dashed,thick,red] (5,-1)--(11,-1);
\draw [thick] (5,-3)--(5,1);
\tdot{5}{1}{black}
\tdot{11}{-1}{blue}
\tdot{5}{-3}{blue}
\node at (5.6,1.5) {$p_1$};
\node at (5.6,-2.6) {$l_1$};
\node at (11.6,-0.6) {$l_2$};
\end{tikzpicture}
\end{center}
\caption{Interacting leaves.} 
\label{fig:inter}
\end{figure}
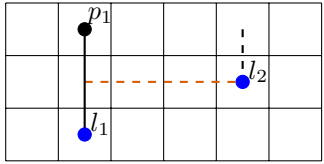

The interest of this notion of interacting leaves is put to light by the following operation.
\begin{definition}\label{def:switch}
For two interacting left leaves $l_1$ and $l_2$ in a CNAT $T$, we define the {\em switch}
of these two leaves as the exchange of the row labels of $l_1$ and $l_2$.
More precisely:
\begin{itemize}
\item we erase $l_1$ and we put a new leaf $l_1'$ in the same column, in row $r(l_2)$;
\item we erase $l_2$ and we put a new leaf $l_2'$ in the same column, in row $r(l_1)$.
\end{itemize}
By doing this, we obtain an object $T'=S(T,l_1,l_2)$.

We have the same operation for right (interacting) leaves.
\end{definition}
\begin{remark}\label{rem:inter-switch}
An easy observation is that after switching, $l_1'$ and $l_2'$ are interacting leaves in $T'$.
\end{remark}
This notion is illustrated by Figure~\ref{fig:switch}.
\begin{figure}[ht]
\begin{center}
\begin{tikzpicture}[scale=0.35]

\draw [step=2] (2,2) grid (14,-6);
\draw [thick] (5,-5)--(5,1);
\draw [thick] (11,-3)--(11,-1);
\tdot{5}{1}{black}
\tdot{11}{-1}{black}
\tdot{5}{-5}{blue}
\tdot{11}{-3}{blue}
\node at (5.6,-4.6) {$l_1$};
\node at (11.6,-2.6) {$l_2$};

\begin{scope}[shift={(22,0)}]
\draw [step=2] (2,2) grid (14,-6);
\draw [thick] (5,-3)--(5,1);
\draw [thick] (11,-5)--(11,-1);
\tdot{5}{1}{black}
\tdot{11}{-1}{black}
\tdot{5}{-3}{blue}
\tdot{11}{-5}{blue}
\node at (5.6,-2.6) {$l'_1$};
\node at (11.6,-4.6) {$l'_2$};
\end{scope}
\end{tikzpicture}
\end{center}
\caption{Switching interacting leaves. 
\label{fig:switch}}
\end{figure}
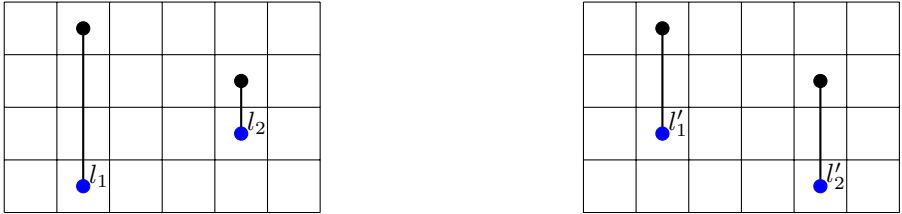
\begin{proposition}\label{prop:switch}
Let $T$ be in $\T_n$.
For two interacting leaves $l_1$ and $l_2$ in $T$, $S(T,l_1,l_2)$ is in $\T_n$.
\end{proposition}
\begin{proof}
The only condition in Definition~\ref{def:nat} that is not trivially satisfied in non ambiguity.
This is a direct consequence of Remark~\ref{rema:alone}.
\end{proof}

The technical part of the construction of $\Phi$ now relies on the two following lemmas.
\begin{lemma}\label{lemm:Phi-crux1}
Consider a CNAT $T$.
Suppose that there is an internal dot which holds as children a leaf and an internal dot.
Then $T$ has at least two interacting leaves.
\end{lemma}
\begin{proof}
We refer to Figure~\ref{fig:Phi-crux1}.
\begin{figure}[ht]
\begin{center}
\begin{tikzpicture}[scale=0.35]

\draw [step=2] (2,2) grid (22,-12);
\draw [dashed,thick] (3,1)--(11,1);
\draw [dashed,thick] (11,-3)--(11,-9)--(17,-9);
\draw [dashed,thick] (17,-9)--(17,-11);
\draw [thick] (11,-1)--(11,1)--(15,1);
\draw [thick] (5,-1)--(5,1);
\draw [thick] (5,1)--(7,1);
\draw [thick] (11,1)--(11,-3);
\draw [thick] (11,-9)--(11,-11);
\draw [thick] (17,-9)--(19,-9);
\tdot{5}{1}{black}
\tdot{7}{1}{black}
\tdot{11}{1}{black}
\tdot{11}{-3}{black}
\tdot{11}{-9}{black}
\tdot{17}{-9}{black}
\tdot{15}{1}{blue}
\tdot{5}{-1}{blue}
\tdot{11}{-11}{blue}
\tdot{19}{-9}{blue}
\node at (5.5,1.5) {$p_1$};
\node at (11.6,1.5) {$p_2$};
\node at (11.6,-8.5) {$p_3$};
\node at (17.6,-8.5) {$p_4$};
\node at (5.6,-0.5) {$l_1$};
\node at (15.6,1.5) {$l_2$};
\node at (11.6,-10.5) {$l_3$};
\node at (19.6,-8.5) {$l_4$};
\end{tikzpicture}
\end{center}
\caption{Proof of Lemma~\ref{lemm:Phi-crux1}.} 
\label{fig:Phi-crux1}
\end{figure}
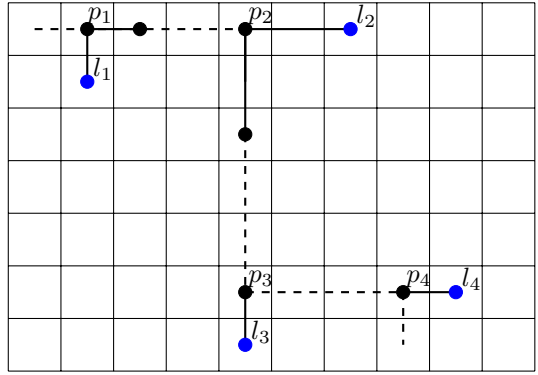
By symmetry, we may assume that we have an internal dot $p_1$ with a left leaf $l_1$ and an internal dot as right child.

Let us suppose that we do not have interacting leaves.
Let us denote by $p_2$ the right-most internal dot in the same row as $p_1$.
Then the right child of $p_2$ has to be a leaf $l_2$.
And its left child has to be an internal dot: if it was a leaf, this leaf would be interacting with $l_1$.
We may now iterate, and consider $p_3$ the bottom-most internal dot in the same column as $p_2$.
For the same reason, $p_3$ holds as children: a left leaf $l_3$ and an internal dot as right child.
We are thus led to an infinite series of internal dots in $T$, which is absurd.

Thus there are at least two intersecting leaves.
\end{proof}

\begin{lemma}\label{lemm:Phi-crux2}
Any CNAT with a long leaf has at least two interacting leaves.
\end{lemma}
\begin{proof}
Consider a CNAT $T$ with a long leaf $l$.
We suppose that $l$ is a left leaf, and denote by $p$ its parent.
Since $l$ is a long leaf, we have: $r(l)-r(p)\ge 2$.
We examine the row with label $r(l)-1$.
This row has to contain a leaf $l'$.
If $l'$ is a left leaf, we are done (this is the case illustrated by Figure~\ref{fig:inter}).
If $l'$ is a right leaf, we call its parent $p'$ (see Figure~\ref{fig:Phi-crux2}).

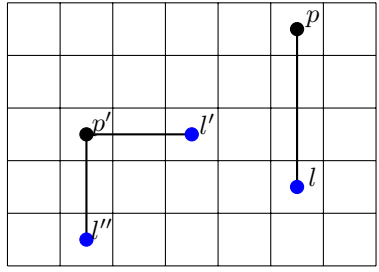
\begin{figure}[ht]
\begin{center}
\begin{tikzpicture}[scale=0.35]

\draw [step=2] (0,0) grid (14,-10);
\draw [thick] (3,-9)--(3,-5)--(7,-5);
\draw [thick] (11,-1)--(11,-7);
\tdot{3}{-5}{black}
\tdot{11}{-1}{black}
\tdot{11}{-7}{blue}
\tdot{7}{-5}{blue}
\tdot{3}{-9}{blue}
\node at (11.6,-0.6) {$p$};
\node at (3.6,-4.6) {$p'$};
\node at (11.6,-6.6) {$l$};
\node at (7.6,-4.6) {$l'$};
\node at (3.6,-8.6) {$l''$};
\end{tikzpicture}
\end{center}
\caption{Proof of Lemma~\ref{lemm:Phi-crux2}.} 
\label{fig:Phi-crux2}
\end{figure}

We have $r(p')=r(l')=r(l)-1$.
If the left child of $p'$ is a leaf $l''$ then $r(l'')>r(l)$ which implies that $l$ and $l''$ are interacting.
And if the left child of $p'$ is an internal dot, we are in the case of Lemma~\ref{lemm:Phi-crux1}
which asserts that $T$ contains two interacting leaves.
\end{proof}

Let us now make precise the construction of $\Phi:\T_n\longrightarrow\T_n$. 
First of all, we define $\Phi(T)=T$ for any $T\in\A_{n}$.
With this done, we are reduced to the case where $T$ has at least one long leaf.
By Lemma~\ref{lemm:Phi-crux2}, $T$ contains interacting leaves.
To define $\Phi$ for such a $T$, we want to {\em choose} a pair of interacting leaves.
Since the set of interacting leaves may change when we switch leaves,
we have to choose in such a way that we create an involution.
If $T$ contains left interacting leaves, we consider the (non-empty) set 
$\{(r(l_1),r(l_2)):\ l_1 \text{ and } l_2 \text{ interacting}\}$
and choose $l_1$ and $l_2$ which correspond to the lexicographical maximum of this set.
Let us call these interacting leaves {\em active}.
This done, we set $\Phi(T)=S(T,l_1,l_2)$.
And if $T$ contains only right interacting leaves, we consider the lexicographical maximum
of $\{(c(l_1),c(l_2));\ l_1 \text{ and } l_2 \text{ interacting}\}$ to choose the pair of active leaves.

\begin{proposition}\label{prop:Phi}
The function $\Phi$ is an involution on $\T_n$.
Moreover, if $T\not\in\A_n$ then $\sgn \pi(\Phi(T)) = -\sgn \pi(T)$.
\end{proposition}
\begin{proof}
By definition of an involution, we want to prove that for any $T\in\T_n$: 
$$\Phi\circ\Phi(T)=T.$$

If $T\in\A_n$, then it is trivial.

Let us consider $T\not\in\A_n$, and set $T'=\Phi(T)$.
We want to prove that if $l_1$ and $l_2$ are the two active leaves in $T$, then $l'_1$ and $l'_2$ are the two active leaves in $T'$.
By Remark \ref{rem:inter-switch}, we have that $l'_1$ and $l'_2$ are interacting 
leaves in $T'$.

By symmetry, we focus on the case where we have left leaves.
We want to prove that if we consider 3 left leaves $l_1$, $l_2$ and $l_3$
such that $r(l_1)>r(l_3)>r(l_2)$ then if $l_1$ and $l_2$ are interacting then $l_1$ and $l_3$ are also interacting.
First of all, since $r(l_2)<r(l_1)$ the fact that $l_1$ and $l_2$ are interacting implies that
$r(p_1)<r(l_2)<r(l_1)$ ($p_1$ denotes the parent of $l_1$).
Whence $r(p_1)<r(l_3)<r(l_1)$, which was to be proved.
This implies that if $l_1$ and $l_2$ are the two active leaves in $T$, then $l'_1$ and $l'_2$ are the two active leaves in $T'$.

This implies that when we apply $\Phi$ to $T'$, we switch $l'_1$ and $l'_2$, and we get:
$\Phi(T')=\Phi\circ\Phi(T)=T$.
Thus $\Phi$ is an involution.

Now the assertion about the determinant comes from the easy observation that
$\pi(\Phi(T))$ and $\pi(T)$ differ by exactly a transposition.
\end{proof}

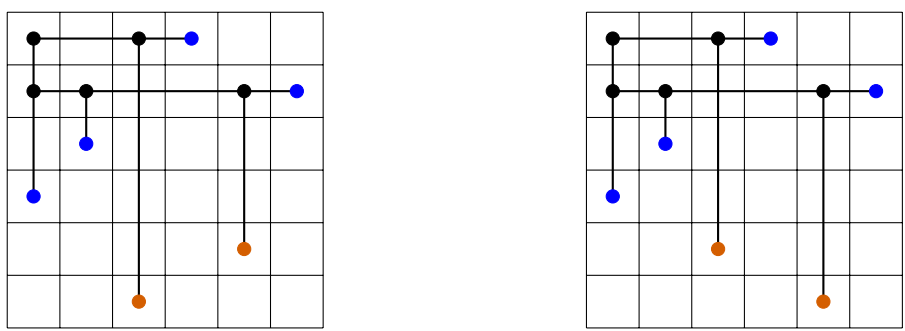
\begin{figure}[hb]
\begin{center}
\begin{tikzpicture}[scale=0.35]

\draw [step=2] (2,2) grid (14,-10);
\draw [thick] (5,-3)--(5,-1)--(13,-1);
\draw [thick] (3,-5)--(3,1)--(9,1);
\draw [thick] (7,1)--(7,-9);
\draw [thick] (3,-1)--(5,-1);
\draw [thick] (11,-1)--(11,-7);
\tdot{3}{1}{black}
\tdot{7}{1}{black}
\tdot{3}{-1}{black}
\tdot{5}{-1}{black}
\tdot{11}{-1}{black}
\tdot{3}{-5}{blue}
\tdot{5}{-3}{blue}
\tdot{7}{-9}{red}
\tdot{9}{1}{blue}
\tdot{11}{-7}{red}
\tdot{13}{-1}{blue}

\begin{scope}[shift={(22,0)}]

\draw [step=2] (2,2) grid (14,-10);
\draw [thick] (5,-3)--(5,-1)--(13,-1);
\draw [thick] (3,-5)--(3,1)--(9,1);
\draw [thick] (7,1)--(7,-7);
\draw [thick] (3,-1)--(5,-1);
\draw [thick] (11,-1)--(11,-9);
\tdot{3}{1}{black}
\tdot{7}{1}{black}
\tdot{3}{-1}{black}
\tdot{5}{-1}{black}
\tdot{11}{-1}{black}
\tdot{3}{-5}{blue}
\tdot{5}{-3}{blue}
\tdot{7}{-7}{red}
\tdot{9}{1}{blue}
\tdot{11}{-9}{red}
\tdot{13}{-1}{blue}

\end{scope}
\end{tikzpicture}
\end{center}
\caption{A CNAT and its image under $\Phi$. Active leaves appear in red. 
\label{fig:Phi}}
\end{figure}

Figure~\ref{fig:Phi} shows an example of the application of $\Phi$.

We can now conclude the proof of our main result.

\begin{proof}[of Theorem~\ref{theo:main}]
It is a consequence of Propositions~\ref{prop:Ap-card}, \ref{prop:Ap-sgn} and \ref{prop:Phi}.
\end{proof}

\printbibliography

\end{document}